\documentclass[12pt,leqno]{article}
\usepackage[doc]{optional}
\usepackage{color}
\usepackage{float}
\usepackage{soul}
\usepackage{graphicx}
\definecolor{labelkey}{rgb}{0,0.08,0.45}
\definecolor{refkey}{rgb}{0,0.6,0.0}
\definecolor{Brown}{rgb}{0.45,0.0,0.05}
\definecolor{lime}{rgb}{0.00,0.8,0.0}
\definecolor{lblue}{rgb}{0.5,0.5,0.99}
\usepackage{mathpazo}
\usepackage{amsmath}
\usepackage{amssymb}
\usepackage{theorem}
\newcommand{\menge}[2]{\big\{{#1}~\big |~{#2}\big\}}

\newcommand{\To}{\ensuremath{\rightrightarrows}}

\newcommand{\fenv}[1]%
{\ensuremath{\,\overrightarrow{\operatorname{env}}_{#1}}}
\newcommand{\benv}[1]%
{\ensuremath{\,\overleftarrow{\operatorname{env}}_{#1}}}

\newcommand{\scal}[2]{\left\langle{#1},{#2}  \right\rangle}

\newcommand{\RR}{\ensuremath{\mathbb R}}
\newcommand{\cR}{\ensuremath{\mathcal R}}

\newcommand{\NN}{\ensuremath{\mathbb N}}
\newcommand{\NON}{\ensuremath{{\mathcal N}(X)}}
\newcommand{\FF}{\ensuremath{{\mathcal F}}}
\newcommand{\CC}{\ensuremath{{\mathcal C}}}
\newcommand{\MM}{\ensuremath{\mathcal M}(X)}
\newcommand{\JJ}{\ensuremath{\mathcal J}(X)}

\newcommand{\dom}{\ensuremath{\operatorname{dom}}}

\newcommand{\gr}{\ensuremath{\operatorname{gr}}}

\newcommand{\inte}{\ensuremath{\operatorname{int}}}

\newcommand{\ran}{\ensuremath{\operatorname{ran}}}

\newcommand{\Fix}{\ensuremath{\operatorname{Fix}}}
\newcommand{\Id}{\ensuremath{\operatorname{Id}}}

\newcommand{\Gc}{\ensuremath{\stackrel{g}{\longrightarrow}}}

\newcommand{\bB}{\ensuremath{{\mathbf{B}}}}

\newcommand{\tro}{\ensuremath{\tilde{\rho}}}
\newcommand{\hro}{\ensuremath{\hat{\rho}}}
\newcommand{\kK}{\ensuremath{\mathcal{K}}}

{\begin{list}{}{%
\settowidth{\labelwidth}{\textrm{#1~}}%
\setlength{\leftmargin}{\labelwidth+\labelsep}}}%requires macro calc.sty
{\end{list}}
\newtheorem{theorem}{Theorem}[section]
\newtheorem{lemma}[theorem]{Lemma}
\newtheorem{corollary}[theorem]{Corollary}
\newtheorem{proposition}[theorem]{Proposition}
\newtheorem{definition}[theorem]{Definition}
\theoremstyle{plain}{\theorembodyfont{\rmfamily}
}
\theoremstyle{plain}{\theorembodyfont{\rmfamily}
}
\theoremstyle{plain}{\theorembodyfont{\rmfamily}
}
\theoremstyle{plain}{\theorembodyfont{\rmfamily}
\newtheorem{example}[theorem]{Example}}
\newtheorem{fact}[theorem]{Fact}
\theoremstyle{plain}{\theorembodyfont{\rmfamily}
\newtheorem{remark}[theorem]{Remark}}

\allowdisplaybreaks % or locally if problems {\allowdisplaybreaks
\begin{document}

\title{\textrm{Most Maximally Monotone Operators  Have \\ a Unique Zero
and a Super-regular Resolvent}}

\author{
%Heinz H.\ Bauschke\thanks{
%Mathematics, University
%of British Columbia,
%Kelowna, B.C.\ V1V~1V7, Canada.
%E-mail: \texttt{heinz.bauschke@ubc.ca}.}
%~and~
Xianfu\ Wang\thanks{Mathematics,
University of British Columbia,
Kelowna, B.C.\ V1V~1V7, Canada.
E-mail:  \texttt{shawn.wang@ubc.ca}.}
}

\date{January 21, 2013}
\maketitle

\vskip 8mm

\begin{abstract} \noindent
Maximally monotone operators play important roles in optimization, variational analysis and differential
equations. Finding zeros of maximally monotone operators has been a central topic. In a Hilbert space, we show that
most resolvents are super-regular, that
most maximally monotone operators have a unique zero and
that the set of strongly monotone mapping is of first category although each strongly monotone operator
has a unique zero. The results
are established
by applying the Baire Category Theorem to the space of nonexpansive mappings.
\end{abstract}

{\small
\noindent
{\bfseries 2010 Mathematics Subject Classification:}
%\hl{needs checking}
{Primary 47H05, 47H10;
Secondary 54E52, 47H09, 54E50.
% 26B25 convexity generalizations? no
% 47H10 fixed point theorems? no
}

}

\section{Introduction}
Throughout, $X$ is a real Hilbert space whose inner product is denoted by
$\langle x,y\rangle$ and induced inner product norm by $\|x\|:=\sqrt{\langle x,x\rangle}$ for
$x,y\in X$. Recall that a set-valued operator $A\colon X\To X$ (i.e., $(\forall x\in
X)$ $Ax\subseteq X$) with {graph} $\gr A$ is \emph{monotone} if
\begin{equation}
(\forall (x,u)\in\gr A)
(\forall (y,v)\in\gr A)
\quad
\scal{x-y}{u-v}\geq 0
\end{equation}
where $\gr A:=\{(x,y)\in X\times X:\ y\in Ax\}$,
and that $A$ is \emph{maximally monotone} if it is impossible to find a
proper extension of $A$ that is still monotone. We call $A:X\To X$
\emph{strongly monotone} \cite{BC2011,Rock98} if there exists $\varepsilon>0$ such that
$A-\varepsilon\Id $ is monotone in which $\Id\colon X\to X\colon x\mapsto x$ denotes
the \emph{identity operator}.

We shall work in the \emph{space of nonexpansive mappings} defined on $X$, i.e.,
$$\NON:=\{T:X\to X:\ \|Tx-Ty\|\leq \|x-y\|, \forall \ x,y\in X\};$$
the \emph{space of firmly nonexpansive mappings}
$$\JJ:=\{T: X\to X:\ \|Tx-Ty\|^2\leq \scal{Tx-Ty}{x-y}, \forall \ x,y\in X\};$$
and the space of \emph{maximal monotone operators}
$$\MM:=\{A:X\To X:\ A \text{ is maximally monotone}\}$$
endowed with a metric defined in Section~\ref{s:generic}.
The reason to investigate nonexpansive mappings defined on $X$ is that they are directly related
to maximally monotone operators.

\emph{In this note, we study generic properties of
$\NON$, $\MM$ and $\JJ$ by the Baire Category Theorem.
%the set of nonexpansive mappings defined on the whole space $X$ and the key tool
A recent result due to Reich and Zaslavski implies that  most nonexpansive mappings in $\NON$ are super-regular
so that each of them has a unique fixed point. Utilizing
Reich and Zaslavaski's technique, we show that
(i) Most resolvents in $\JJ$
are super regular, thus asymptotically regular; (ii) Most maximally monotone operators in $\MM$ have a unique zero;
(iii) The set of strongly monotone operators is only a first category set in $\MM$ even though it is dense.}

While extensive study has been done
on $\NON$ \cite{BC2011,GR,GK,Blasi,Blasi76,reichcom,reichmath00,zaslavski00,zaslavski01} and
on $\MM$ \cite{attouch,BC2011,Rock98,Simons2,Zeidler2b}, generic properties on $\MM$ and $\JJ$
seem new. They are particularly interesting for the optimization field. Note that De Blasi and Myjak
only considered generic properties of continuous bounded monotone operators on a bounded set in \cite{Blasi79}.

In the reminder of this section, we introduce some definitions, basic facts and preliminary results.
For $A\in\MM$, we define its \emph{resolvent} and \emph{reflected resolvent} (or Cayley transform)
by
$$J_{A}:=(A+\Id)^{-1},\quad R_{A}:=2J_{A}-\Id.$$
 It is well-known that
$J_{A}+J_{A^{-1}}=\Id$, $R_{A}+R_{A^{-1}}=0$, see, e.g., \cite{Rock98}, \cite[Proposition 4.1]{moursi12}.
Both resolvent and reflected resolvent play a key role in the proximal point algorithm and Douglus-Rachford algorithm \cite{BC2011,rockprox,combettes02,Comb04,lions,moursi12}.

The following well-known characterizations about firmly nonexpansive mappings, nonexpansive mappings and
maximally monotone operators are crucial.

\begin{fact}
\label{f:firm}
{\rm (See, e.g., \cite{BC2011,GK,GR}.)}
Let $T\colon X\to X$.
Then the following are equivalent:
\begin{enumerate}
\item $T$ is firmly nonexpansive.
\item $2T-\Id$ is nonexpansive.
\item $(\forall x\in X)(\forall y\in X)$
$\|Tx-Ty\|^2 \leq \scal{x-y}{Tx-Ty}$.
\item
\label{f:firmsymm}
$(\forall x\in X)(\forall y\in X)$
$0\leq \scal{Tx-Ty}{(\Id-T)x-(\Id-T)y}$.
\end{enumerate}
\end{fact}

\begin{fact}[Eckstein \& Bertsekas, Minty]\emph{\cite{Minty,Zeidler2b,eckstein92}} \label{mintymono}
 Let $A:X\To X$ be monotone. Then $A$ is maximally monotone if and only if
$J_{A}$ is firmly non-expansive and
has a full domain.
\end{fact}
For $T:X\rightarrow X$, let $\Fix T$ denote its fixed point set
$\Fix T: =\{x\in X:\ Tx=x\}$.
Facts~\ref{f:firm}, \ref{mintymono} allow us to summarize the relationship among $\NON, \JJ, \MM$.
\begin{proposition}\label{p:relation}
\begin{enumerate}
\item $\NON=\{R_{A}: A\in \MM\},$ $$\MM=\left\{\bigg(\frac{T+\Id}{2}\bigg)^{-1}-\Id:\ T\in \NON\right\}.$$
\item $\JJ=\{J_{A}:\  A\in \MM\},$
$$\MM=\left\{T^{-1}-\Id:\ T\in\JJ\right\}.$$
\item $\NON=\{2T-\Id: \ T\in \JJ\},$
$$\JJ=\bigg\{\frac{T+\Id}{2}:\ T\in \NON\bigg\}.$$
\item Let $A\in \MM$. Then $\Fix R_{A}=\Fix J_{A} =A^{-1}(0).$
\end{enumerate}
\end{proposition}

Many nice properties and applications about $\NON, \JJ, \MM$ can be found in \cite{attouch,BC2011,bmw12,zaslavski00,zaslavski01}
and they
continue to flourish. We refer readers to \cite{bmw12} for a systematic relationship among these three spaces.
Let us now turn to the graphical convergence of set-valued maximal monotone operstors.

\begin{definition}\emph{\cite[page 360]{attouch}} Given a sequence of maximally monotone
operators $$\{A_{n}:\ n\in\NN\}, A.$$ The sequence $\{A_{n}:\ n\in\NN\}$ is said to be graphically
convergent to $A$, written as
$A_{n}\Gc A$, if
\begin{quote}
for every $(x,y)\in \gr A$ there exists $(x_{n},y_{n})\in \gr A_{n}$ such that
$x_{n}\rightarrow x, y_{n}\rightarrow y$ strongly in $X\times X$.
\end{quote}
In terms of set convergence $\gr A \subset \liminf \gr A_{n}$.
\end{definition}

\begin{proposition}\label{p:graphical}
The following are equivalent
\begin{enumerate}
\item\label{i:seqmono} A sequence of maximally monotone operators $(A_{k})_{k=1}^{\infty}$ in $\MM$ converges
graphically to $A$;
\item\label{i:seqresol} $(J_{A_k})_{k=1}^{\infty}$ converges pointwise to $J_{A}$ on $X$;
\item\label{i:seqcaley} $(R_{A_{k}})_{k=1}^{\infty}$ converges pointwise to $R_{A}$ on $X$.
\end{enumerate}
\end{proposition}
\begin{proof}
\ref{i:seqmono}$\Leftrightarrow$ \ref{i:seqresol} follows from \cite[Proposition 3.60, pages 361-362]{attouch}.
\ref{i:seqresol}$\Leftrightarrow$ \ref{i:seqcaley} is obvious since $R_{A}=2J_{A}-\Id$.
\end{proof}

A set $S$ in a complete metric space $Y$ is called \emph{residual} if there is a sequence of dense and open sets
$O_{n}\subset Y$ such that $\bigcap_{n=1}^{\infty}O_{n}\subset S$; in this case we call $\bigcap_{n=1}^{\infty}O_{n}$ a
dense $G_{\delta}$ set. A classical theorem of Baire is
\begin{fact}[Baire Category Theorem]\emph{\cite[page 158]{royden}}\label{f:baire} Let $Y$ be a complete metric space and $\{O_{n}\}$ a countable
collection of dense open subsets of $Y$. Then $\bigcap_{n=1}^{\infty}O_{n}$ is dense in $Y$.
\end{fact}
The technique of Baire Category has been instrumental in studying fixed point of nonexpansive mappings; see, e.g., \cite{Blasi,Blasi79,Blasi76,reichcom,reichza01,reichmath00,zaslavski00,zaslavski01}.

The paper is organized as follows. In Section~\ref{s:generic} we give the main result.
In Section~\ref{s:weak} we introduce a class of weakly contractive mappings which contains contractive mappings, and show that
although it is dense, it is only a set of first category.

{\bf Notation.} For a set-valued mapping $A:X\To X$,
we write $\dom A :=\menge{x\in X}{Ax\neq\varnothing}$
and $\ran A := A(X) = \bigcup_{x\in X}Ax$ for the
\emph{domain} and \emph{range} of $A$, respectively.
$\bB_{r}(x)$ denotes the closed ball of radius $r$ centered at $x$.
$\NN$ stands for the set of natural numbers.

\section{Main results}\label{s:generic}
%We start with Three complete metric spaces}
In this section, using Reich and Zaslavski's technique on super-regular mappings we establish
a generic property of super-regular mappings in complete subspaces of $(\NON,\rho)$. This allows us to show
that most resolvents are super-regular; most maximally monotone operators
have a super-regular reflected resolvent
and a unique zero.

We start with three complete metric spaces
which set up the stage for the Baire Category Theorem.

On $\NON$ we define a metric, for $T_{1}, T_{2}\in \NON$
\begin{equation}\label{e:metricnonexp} \rho(T_{1},T_{2}):=\sum_{n=1}^{\infty}\frac{1}{2^n}\frac{\|T_{1}-T_{2}\|_{n}}{1+\|T_{1}-T_{2}\|_{n}}
\end{equation}
where $\|T_{1}-T_{2}\|_{n}:=\sup_{\|x\|\leq n}\|T_{1}x-T_{2}x\|.$
The metric $\rho$ defines a topology of
pointwise convergence on $X$ and uniform convergence on bounded subsets of $X$.

\begin{proposition}\label{p:setup}
%\begin{enumerate}
%\item
$(\NON,\rho)$ is a complete metric space.
%\item $(\JJ,\rho)$ is a complete metric space.
%\end{enumerate}
\end{proposition}
\begin{proof}
It is easy to see that $\rho$ is a metric (cf. \cite[pages 10-11]{kreyzig}). We show that $\NON$ is complete.
Assume that $(T_{k})_{k=1}^{\infty}$ is a Cauchy sequence in $(\NON, \rho)$.
Then for every $n\in \NN$,
$(T_{k})_{k=1}^{\infty}$
is a uniform Cauchy sequence on $\bB_{n}(0)$. In particular, $(T_{k}(x))_{k=1}^{\infty}$ is
Cauchy in $X$ for each $x\in \bB_{n}(0)$, so $T_k(x)$ converges to $Tx\in X$. Moreover, for every
$n\in\NN$, $\|T_{k}-T\|_{n}\rightarrow 0$ as $k\rightarrow\infty$.
Since each $T_{k}$ is nonexpansive,
$T$ is nonexpansive, i.e., $T\in \NON$. It remains to show $\rho(T_{k},T)\rightarrow 0$ as $k\rightarrow\infty$.
Let $\varepsilon>0$. Choose $M\in\NN$ large such that
$$\sum_{n=M+1}^{\infty}
\frac{1}{2^{n}}<\frac{\varepsilon}{2}.$$
For this $M$, choose a large $N\in \NN$ such that $\|T_{k}-T\|_{M}<\frac{\varepsilon}{2}$ when
$k>N$.
Then for $k>N$ we have
\begin{align}
\rho(T_{k},T) &=\sum_{n=1}^{M}\frac{1}{2^n}\frac{\|T_{k}-T\|_{n}}{1+\|T_{k}-T\|_{n}}
+\sum_{n=M+1}^{\infty}\frac{1}{2^n}\frac{\|T_{k}-T\|_{n}}{1+\|T_{k}-T\|_{n}}\\
&\leq \sum_{n=1}^{M}\frac{1}{2^n}\frac{\varepsilon/2}{1+\varepsilon/2}+
\sum_{n=M+1}^{\infty}\frac{1}{2^n}<\varepsilon/2+\varepsilon/2=\varepsilon.
\end{align}
Hence $(\NON,\rho)$ is complete.
\end{proof}
\begin{remark} In \cite{reichmath00}, Reich and Zaslavaski define a unform space
$(\NON, {\cal U})$ where the uniformity ${\cal U}$ is defined by
the base
$$E(n,\varepsilon)=\{(T,S)\in\NON\times \NON:\ \|T-S\|_{n}<\varepsilon\}$$
for $n\in\NN, \varepsilon>0$.
The topology induced by this uniformity and the metric $\rho$ are exactly the same.
\end{remark}

On $\MM$ let us define a metric
\begin{equation}\label{e:metricmono}
\tro(A,B):=\rho(R_{A}, R_{B})=\sum_{n=1}^{\infty}\frac{1}{2^n}\frac{\|R_{A}-R_{B}\|_{n}}{1+\|R_{A}-R_{B}\|_{n}}
\end{equation}
for $A, B\in \MM$.
\begin{proposition}\label{p:mcomp}
\begin{enumerate}
\item The space of monotone operators $(\MM,\tro)$ is a complete metric space, and it is isometric
to $(\NON,\rho)$.
\item When $X=\RR^{N}$, the topology on $(\MM,\tro)$ is precisely the topology of
graphical convergence.
\end{enumerate}
\end{proposition}
\begin{proof}
(i) By Facts~\ref{f:firm}, \ref{mintymono}, under the mapping $A\mapsto R_{A}$
$$(\MM,\tro)\text{ and } (\NON,\rho) \text{ are isometric}.$$
Since $(\NON,\rho)$ is complete by Proposition~\ref{p:setup}, we conclude that
$(\MM,\tro)$ is complete.

(ii) When $X=\RR^N$, on $\NON$ pointwise convergence and uniform convergence on compact
subsets are the same. By Proposition~\ref{p:graphical}, we obtain that the topology on $(\MM,\tro)$ is exactly
the topology of graphical convergence.
\end{proof}

On $\JJ$ let us define a metric
\begin{equation}\label{e:metricresol}
\hro(T_{1},T_{2}):=\rho(2T_{1}-\Id, 2T_{2}-\Id)=\sum_{n=1}^{\infty}\frac{1}{2^n}\frac{\|2T_{1}-2T_{2}\|_{n}}{1+\|2T_{1}-2T_{2}\|_{n}}
\end{equation}
for $T_{1},T_{2}\in \JJ$.
\begin{proposition} The space of resolvents
$(\JJ,\hro)$ is a complete metric space, and it is isometric to $(\NON, \rho)$.
\end{proposition}
\begin{proof}
By Fact~\ref{f:firm}, under the mapping $T\mapsto 2T-\Id$
\begin{equation}\label{e:isometry}
(\JJ,\hro)\text{ and } (\NON,\rho) \text{ are isometric}.
\end{equation}
Since $(\NON,\rho)$ is complete by Proposition~\ref{p:setup}, the result holds.
\end{proof}

Next we study the denseness of contraction mappings and strongly monotone operators, which are required
in later proofs.
\begin{definition} The map
 $T\in\NON$ is called a contraction with modulus $1>l\geq 0$ if
$$\|Tx-Ty\|\leq l\|x-y\|\  \forall \ x,y\in X.$$
\end{definition}
\begin{lemma}\label{l:cdense}
\begin{enumerate}
\item ({\bf denseness of contraction mappings}) In $(\NON,\rho)$ the set of contractions is dense, i.e., for very $\varepsilon>0$ and
$T\in\NON$ there exists a contraction $T_{1}\in \NON$ such that
$\rho(T,T_{1})<\varepsilon$.
\item ({\bf denseness of contractive firmly nonexpansive mappings}) In $(\JJ,\hro)$ the set of contraction is dense, i.e.,
for very $\varepsilon>0$ and
$T\in\JJ$ there exists a contraction $T_{1}\in \JJ$ such that
$\hro(T,T_{1})<\varepsilon$.
\end{enumerate}
\end{lemma}
\begin{proof} (i) Let $T\in \NON$ and $1>\varepsilon>0$.
Choose an integer $M$ sufficiently large such that
\begin{equation}\label{e:M}
\sum_{n= M+1}^{\infty}\frac{1}{2^{n}}\leq \frac{\varepsilon}{2}.
\end{equation}
%For this $M$, define
%$$T_{1}:=T\circ P_{{\bB_{M}(0)}}$$
%where $P_{\bB_{M}(0)}$ denotes the projection operator on the closed ball $\bB_{M}(0)$.
%Note that $T_{1}$ is nonexpansive since $T$ and $P_{\bB_{M}(0)}$ are nonexpansive, and moreover $\ran T_{1}= T(\bB_{M}(0))$ is a bounded set.
Choose
$$0<\lambda <\frac{\varepsilon}{2(1+\|T\|_{M})}<\frac{1}{2}$$
and define
$$T_{1}:=(1-\lambda)T.$$
Then $T_{1}$ is a contraction with modulus $1/2<1-\lambda<1$.
As
\begin{align}
\|T_{1}-T\|_{M} & =\sup_{\|x\|\leq M}\|(1-\lambda)Tx-Tx\|\\
&=\lambda\sup_{\|x\|\leq M}\|Tx\|=\lambda\|T\|_{M}<\frac{\varepsilon}{2}.
\end{align}
%By the definition of $\rho$,
Using $\displaystyle \|T_{1}-T\|_{n}\leq \|T_{1}-T\|_{M}<\frac{\varepsilon}{2}$ when $n\leq M$
and \eqref{e:M}, we have
\begin{align}
\rho(T_{1},T)&=\sum_{n=1}^{M}\frac{1}{2^n}\frac{\|T_{1}-T\|_{n}}{1+\|T_{1}-T\|_{n}}
+\sum_{n=M+1}^{\infty}\frac{1}{2^n}\frac{\|T_{1}-T\|_{n}}{1+\|T_{1}-T\|_{n}}\\
&\leq \sum_{n=1}^{M}\frac{1}{2^n}\frac{\varepsilon}{2}+\sum_{n=M+1}^{\infty}\frac{1}{2^{n}}\\
&< \frac{\varepsilon}{2}+\frac{\varepsilon}{2}=\varepsilon
\end{align}
so $\rho(T, T_{1})<\varepsilon$.

(ii) The proof is similar as in (i) by replacing $\rho$ by $\hro$ and by observing that
$T_{1}=(1-\lambda) T\in\JJ$ if $T\in \JJ$ and $0\leq \lambda \leq 1$.
\end{proof}

To study monotone operators, we need:

\begin{fact}\emph{\cite[Corollary 4.7]{bmw12}}\label{f:monocontr}
Let $A:X\To X$ be maximally monotone. Then the following are equivalent:
\begin{enumerate}
\item Both $A$ and $A^{-1}$ are strongly monotone;
\item There exists $\varepsilon>0$ such that both $(1+\varepsilon)J_{A}$ and $(1+\varepsilon)J_{A^{-1}}$
are firmly nonexpansive;
\item $R_{A}$ is a Banach contraction.
\end{enumerate}
\end{fact}
\begin{lemma}[denseness of strongly monotone mappings]
In $(\MM,\tro)$ the set of monotone operators $A$ such that both $A$ and $A^{-1}$ are strongly
 monotone is dense, i.e., for very $\varepsilon>0$ and
$A\in\MM$ there exists a $B\in \MM$ such that both $B$ and $B^{-1}$ are strongly monotone, and
$\tro(A,B)<\varepsilon$. Consequently, the set of strongly monotone
operators is dense in $\MM$.
\end{lemma}
\begin{proof}
Under the mapping $A\mapsto R_{A}$
$$(\MM,\tro)\text{ and } (\NON,\rho) \text{ are isometric}.$$
Let $A\in\MM$ and $\varepsilon>0$. By Lemma~\ref{l:cdense}(i), for $R_{A}$ there exists a contraction $T_{1}$ such that
$\rho(R_{A},T_{1})<\varepsilon$. Proposition~\ref{p:relation}(i) says that there exists
$B\in\MM$ such that $T_{1}=R_{B}$. By Fact~\ref{f:monocontr} both $B, B^{-1}$ are strongly monotone.
The proof is complete by using $\tro(A,B)=\rho(R_{A},R_{B})$.
\end{proof}

To prove our main results, we require \emph{super-regular mappings} introduced by Reich and Zaslavaski \cite{reichmath00}.
\begin{definition}[Reich-Zaslavski] A mapping $T:X\to X$ is called super-regular if there exists a
unique $x_{T}\in X$ such that
for each $s>0$, when $n\rightarrow\infty$,
$$T^{n}x\rightarrow x_{T}\quad \text{ uniformly on } \bB_{s}(0).$$
\end{definition}

Our next two results collect some elementary properties of super-regular mappings.
\begin{proposition}\label{p:superfix} Assume that $T:X\to X$ is super-regular and continuous. Then
$\Fix{T}$ is a singleton.
\end{proposition}

\begin{proof} Let $x\in X$.
Using the continuity and super-regularity of $T$, we have
$$x_{T}=\lim_{n\rightarrow\infty}T^{n}x=\lim_{n\rightarrow\infty}T(T^{n-1}x)=Tx_{T}$$
so $x_{T}\in\Fix{T}$. Let $x\in\Fix T$. By the super-regularity of $T$ and $T^{n}x=x$,
$x=\lim_{n\rightarrow\infty}T^n x=x_{T}.$ Hence $\Fix T=\{x_{T}\}$.
\end{proof}

\begin{proposition}\label{p:contract}
 \begin{enumerate}
 \item If $T\in \NON$ is a contraction, then
$T$ is super-regular.
\item If $A\in \MM$ has both $A$ and $A^{-1}$ being strongly monotone, then $R_{A}$ and
$J_{A}$ are super-regular.
\end{enumerate}
\end{proposition}
\begin{proof} (i) Let $s>0$. Let $T$ be a contraction with modulus $0\leq l<1$. By the Banach Contraction Principle
\cite[pages 300-302]{kreyzig},
$T$ has a unique fixed point $x_{T}$, and with arbitrary $x\in X$ the error estimate is
$$\|T^{n}x-x_{T}\|\leq \frac{l^{n}}{1-l}\|x-Tx\|.$$
For every $x\in \bB_{s}(0)$,
$$\|T^{n}x-x_{T}\|\leq \frac{l^{n}}{1-l}(\|x\|+\|Tx-T0\|+\|T0\|)\leq \frac{l^{n}}{1-l}(s+ls+\|T0\|) .$$
Therefore,
$$\|T^{n}-x_{T}\|_{s}\leq \frac{l^{n}}{1-l}(s+ls+\|T0\|)\rightarrow 0\quad \text{ when $n\rightarrow\infty$}.$$
Since $s>0$ was arbitrary, $T$ is super-regular.

(ii) By Fact~\ref{f:monocontr}, $R_{A}$ is a contraction. Since $A$ is strongly monotone, $J_{A}$ is
a contraction \cite{rockprox}. Hence (i) applies.
\end{proof}

The proof ideas to Proposition~\ref{p:keyresult} and Theorem~\ref{t:unified} below
 are due to Reich and Zaslaski \cite{reichmath00,reichza01}. We adopt them to our complete
 metric space setting, and to subspaces of $\NON$.
For two metrics $\rho, d$ on $\FF\subset \NON$, if $\rho(T_{1},T)\leq d(T_{1},T)$ for all $T_{1},T\in\FF$
we write $\rho\leq d$.

\begin{proposition}\label{p:keyresult} Assume that $\FF \subseteq\NON$, $(\FF,d)$ is complete
and $d\geq \rho$.
Let $T\in \FF$ be super-regular and $\varepsilon, s$ be positive numbers. Then there exists
$\delta>0$ and $n_{0}\geq 2$ such that when $d(T_{1},T)<\delta$
and $n\geq n_{0}$ we have
\begin{equation}\label{e:iterates}
\|T_{1}^{n}x-x_{T}\|< \varepsilon\quad \text{ for every } x\in \bB_{s}(0),
\end{equation}
i.e., $\|T_{1}^{n}-x_{T}\|_{s}<\varepsilon.$
\end{proposition}
\begin{proof} We may and do assume that $0<\varepsilon<1/2$. Let $x_{T}$ denote the
unique fixed point of $T$.
Choose an integer $M>1+2s+4\|x_{T}\|$ so that
\begin{equation}\label{e:parameter}
s<\frac{M}{2},\qquad \frac{1}{2}+s+2\|x_{T}\|<\frac{M}{2}.
\end{equation}
As $T$ is super-regular, there exists $n_{0}\geq 2$ such that
\begin{equation}\label{e:regular}
\|T^{n}x-x_{T}\|<\frac{\varepsilon}{8} \quad \text{ whenever $x\in \bB_{M}(0)$
and $n\geq n_{0}$.}
\end{equation}
Put
$$\delta:=\frac{1}{2^M}\bigg(\frac{(8n_{0})^{-1}\varepsilon}{1+(8n_{0})^{-1}\varepsilon}\bigg).$$
We will show that \eqref{e:iterates} holds when $d(T_{1},T)<\delta$ and $n\geq n_{0}$.

Let $d(T_{1},T)<\delta$. Then $\rho(T_{1},T)<\delta$. Using the definition of $\rho$ and that
$t\mapsto \frac{t}{1+t}$ is strictly increasing on $[0,+\infty)$, we have
\begin{equation}\label{e:neighbor}
\|T_{1}-T\|_{M}<(8n_{0})^{-1}\varepsilon.
\end{equation}

{\noindent \sl Claim 1.} Whenever $x\in \bB_{M/2}(0)$ and $1\leq n\leq n_{0}$,
\begin{align}
\|T_{1}^{n}x-T^{n}x\| & < n (8n_{0})^{-1}\varepsilon, \label{e:map}\\
\|T_{1}^{n}x\| & <\frac{1}{2}+\|x\|+2\|x_{T}\|<M. \label{e:mappower}
\end{align}
We prove this by induction.
As $T\in\NON$ and $Tx_{T}=x_{T}$, for every $n\in \NN$,
\begin{align}\label{e:tpowerdiff}
\|T_{1}^{n}x-T^{n}x\| & \leq \|T_{1}^{n}x-TT_{1}^{n-1}x\|+\|TT_{1}^{n-1}x-T^{n}x\|\\
&\leq \|T_{1}^{n}x-TT_{1}^{n-1}x\|+\|T_{1}^{n-1}x-T^{n-1}x\|,
\end{align}
and
\begin{align}\label{e:tpowerfix}
\|T_{1}^{n}x-x_{T}\| &\leq \|T_{1}^{n}x-T^{n}x\|+\|T^{n}x-x_{T}\|\\
&\leq \|T_{1}^{n}x-T^{n}x\|+\|x-x_{T}\|\\
&\leq \|T_{1}^{n}x-T^{n}x\|+\|x\|+\|x_{T}\|.
\end{align}
Now when $n=1$, \eqref{e:map} follows from \eqref{e:neighbor}; for \eqref{e:mappower},
by \eqref{e:tpowerfix} and \eqref{e:neighbor}
$$\|T_{1}x\|\leq \|T_{1}x-x_{T}\|+\|x_{T}\|\leq \|T_{1}x-Tx\|+\|x\|+2\|x_{T}\|<
\frac{1}{2}+\|x\|+2\|x_{T}\|.$$
Assume that \eqref{e:map}-\eqref{e:mappower} hold for $1\leq n<n_{0}$, i.e.,
\begin{align}
\|T_{1}^{n}x-T^{n}x\| & < n (8n_{0})^{-1}\varepsilon, \label{e:map1}\\
\|T_{1}^{n}x\| & <\frac{1}{2}+\|x\|+2\|x_{T}\|<M. \label{e:mappower1}
\end{align}
 Using \eqref{e:tpowerdiff} for $n+1$, \eqref{e:map1}, \eqref{e:neighbor},
$\|T_{1}^{n}x\|<M$ and $n<n_{0}$,  we obtain
\begin{align}\label{e:iterated}
\|T_{1}^{n+1}x-T^{n+1}x\| &\leq \|T_{1}^{n+1}x-TT_{1}^{n}x\|+\|T_{1}^{n}x-T^{n}x\|\\
&< (8n_{0})^{-1}\varepsilon +n(8n_{0})^{-1}\varepsilon=(n+1)(8n_{0})^{-1}\varepsilon.
\end{align}
Using \eqref{e:tpowerfix} for $n+1$, \eqref{e:iterated},
\begin{align}
\|T_{1}^{n+1}x\| &\leq \|T_{1}^{n+1}x-x_{T}\|+\|x_{T}\|\\
&\leq \|T_{1}^{n+1}x-T^{n+1}x\|+\|x\|+2\|x_{T}\|\\
&<(n+1)(8n_{0})^{-1}\varepsilon +\|x\|+2\|x_{T}\|<\frac{1}{2} +\|x\|+2\|x_{T}\|.
\end{align}
This establishes \eqref{e:map}-\eqref{e:mappower}.

{\noindent \sl Claim 2.}
\begin{equation}\label{e:t1power}
\|T_{1}^{n}y-x_{T}\|<\varepsilon\quad\ \text{ whenever $y\in\bB_{s}(0)$ and $n\geq n_{0}$}.
\end{equation}
This is done again by induction. When $n=n_{0}$, as $\|y\|\leq s<M/2$, by \eqref{e:regular}
and \eqref{e:map}
$$\|T_{1}^{n_{0}}y-x_{T}\|\leq \|T_{1}^{n_{0}}y-T^{n_{0}}y\|+\|T^{n_{0}}y-x_{T}\|
<\varepsilon/8+\varepsilon/8<\varepsilon.$$
Assume that \eqref{e:t1power} holds for all $n_{0}\leq n\leq k$. For $i=1,\ldots, n_{0}$,
\eqref{e:mappower} and \eqref{e:parameter} give
\begin{equation}\label{e:smallcase}
\|T_{1}^{i}y\|<1/2+\|y\|+2\|x_{T}\|<1/2+s+2\|x_{T}\|<M/2;
\end{equation}
For $k\geq i>n_{0}$, \eqref{e:t1power} and \eqref{e:parameter} give
\begin{equation}\label{e:biggercase}
\|T_{1}^{i}y\|\leq \|T_{1}^{i}y-x_{T}\|+\|x_{T}\|<1/2+\|x_{T}\|<M/2.
\end{equation}
Set $j=k+1-n_{0}$ and $x=T_{1}^{j}y$.
Then $1\leq j<k$ and $\|x\|<M/2$ by \eqref{e:smallcase} and \eqref{e:biggercase}.
Combining \eqref{e:regular}, \eqref{e:map} and \eqref{e:mappower} yields
\begin{align}
\|T_{1}^{k+1}y-x_{T}\| & = \|T_{1}^{n_{0}}x-x_{T}\|
\\
&\leq \|T_{1}^{n_{0}}x-T^{n_{0}}x\|+\|T^{n_{0}}x-x_{T}\|<\varepsilon/8+\varepsilon/8
<\varepsilon.
\end{align}
This completes the proof.
\end{proof}

Our first main result comes as follows.
\begin{theorem}[generic property of super-regular mappings in complete subspaces]\label{t:unified}
Let $(\FF,d)$ be a complete metric space, $\FF\subset \NON$ and $d\geq \rho$. Assume that the set of contraction mappings
$\CC$ is dense in $\FF$. Then there exists a set
$G\subset \FF$ which is a countable intersection of open everywhere dense set in $\FF$ such that
each $T\in G$ is super-regular. In particular, $\Fix(T)=(\Id-T)^{-1}(0)\neq\varnothing$ is a singleton.
\end{theorem}
\begin{proof}
%Let $\CC$ be the set of all contraction mappings in $\FF$.
 By Proposition~\ref{p:keyresult} and Proposition~\ref{p:contract}(i),
for each $T\in \CC$, in $(\FF,d)$ there exists an open
neighborhood
$U(T,i)$ of $T$ and an integer $n(T,i)\geq 2$ such that whenever $T_{1}\in U(T,i)$, $n\geq n(T,i)$ and
$x\in\bB_{i}(0)$
\begin{equation}\label{e:ineighbor}
\|T_{1}^{n}x-x_{T}\|<\frac{1}{i}.
\end{equation}
Define $G:=\bigcap_{q=1}^{\infty}O_{q}$ where
$$O_{q}:=\bigcup\{U(T,i):\ T\in \CC, i=q, q+1, \ldots\}$$
which is dense and open in $\FF$, since $\CC\subset O_{q}$ and each $U(T,i)$ is open.

Let $T\in G$. Then there exists a sequence $(T_{q})_{q=1}^{\infty}$ and a
sequence $(i_{q})_{q=1}^{\infty}$ with $i_{q}\geq q$ such that $T\in U(T_{q}, i_{q})$ for $q=1,2,\ldots$.
Then for each $q$, by \eqref{e:ineighbor}, when $n\geq n(T_{q},i_{q})$ and $x\in \bB_{i_{q}}(0)$ we have
\begin{equation}\label{e:specialseq}
\|T^{n}x-x_{T_{q}}\|<\frac{1}{i_{q}}.
\end{equation}
It follows that when $n\geq \max\{n(T_{q},i_{q}), n(T_p,i_{p})\}$ and $\|x\|\leq\min\{i_{p}, i_{q}\}$,
$$\|x_{T_{q}}-x_{T_{p}}\|\leq \|x_{T_{q}}-T^{n}x\|+\|T^n x-x_{T_{p}}\|<\frac{1}{i_{q}}+\frac{1}{i_{p}},$$
thus  $(x_{T_{q}})_{q=1}^{\infty}$ is a Cauchy sequence with a limit $x_{T}\in X$.
Let $s>0$ and $\varepsilon>0$. Choose $i_q$ and $q$ sufficiently large such that $\bB_{s}(0)\subset \bB_{i_{q}}(0)$ and
$$\frac{1}{i_{q}}+\|x_{T_{q}}-x_{T}\|<\varepsilon.$$
In view of \eqref{e:specialseq}, for every
$x\in \bB_{s}(0)$ and $n\geq n(T_{q},i_q)$ we have
$$\|T^{n}x-x_{T}\|\leq \|T^{n}x-x_{T_{q}}\|+\|x_{T_{q}}-x_{T}\|<\frac{1}{i_{q}}+\|x_{T_{q}}-x_{T}\|<\varepsilon.$$
Hence $T$ is super-regular. The remaining result follows from Propostion~\ref{p:superfix}.
\end{proof}

Different choices of $\FF$ lead to:
\begin{theorem}[Reich \& Zaslavaski \cite{reichmath00}]\label{t:reichz}
 There exists a set $G\subset \NON$ which is a countable intersection of open
everywhere dense sets in $\NON$ such that each $T\in G$ is super-regular.
\end{theorem}
\begin{proof}
By Lemma~\ref{l:cdense}(i), the set of contractions
$\CC\subset\NON$ is dense in $\NON$. Apply Theorem~\ref{t:unified} to the complete metric space
$(\NON,\rho)$.
\end{proof}

\begin{theorem}[super-regularity of resolvents]\label{t:resol}
 In $(\JJ,\hro)$, the set
$$\{T\in\JJ:\ T \text{ is super-regular}\}$$
is residual.
\end{theorem}
\begin{proof}
By Lemma~\ref{l:cdense}(ii), the set of contractions
$\CC\subset \JJ$ is dense in $\JJ$. Since
$$\hro(T_{1},T_{2})=\sum_{n=1}^{\infty}\frac{1}{2^n}\frac{2\|T_{1}-T_{2}\|_{n}}{1+2\|T_{1}-T_{2}\|_{n}}
\geq \sum_{n=1}^{\infty}\frac{1}{2^n}\frac{\|T_{1}-T_{2}\|_{n}}{1+\|T_{1}-T_{2}\|_{n}}=\rho(T_{1},T_{2})
\quad \forall \ T_{1}, T_{2}\in \JJ$$
by \eqref{e:metricnonexp} and \eqref{e:metricresol},
we have $\hro\geq\rho.$ It remains to apply Theorem~\ref{t:unified} to the complete metric space
$(\JJ,\hro)$.
\end{proof}

Finding zeros of maximally monotone operators
are important in optimization; see, e.g., \cite{BC2011,Comb04,Sabach11,lions,rockprox}.
However, we have

\begin{theorem}[unique zero of monotone operators]\label{t:mono}
In $(\MM,\tro)$ there is a set $G\subset \MM$ which is a countable intersection of open everywhere
dense sets in $\MM$ such that each $A\in G$ has
$R_{A}$ super-regular. In particular, $A^{-1}(0)\neq\varnothing$ is a singleton.
\end{theorem}
\begin{proof} By
Proposition~\ref{p:mcomp}, $(\MM,\tro)$ is isometric to
$(\NON,\rho).$ Apply Theorem~\ref{t:reichz} to $(\NON,\rho)$ to obtain $\tilde{G}$ such that each $T\in \tilde{G}$
 is super-regular and $\tilde{G}$ is a countable intersection of
 open everywhere dense sets in $\NON$. This $\tilde{G}$ corresponds to $G$ in $(\MM,\tro)$ such that each
$A\in G$ has $R_{A}$ being super-regular and $G$ is an intersection of open everywhere dense set
in $\MM$. Note that $\Fix (R_{A})=A^{-1}(0)$
by Proposition~\ref{p:relation}(iv).
Since $\Fix (R_{A})$ is a singleton when $A\in G$ by Proposition~\ref{p:superfix},
the result holds.
\end{proof}

In this connection, see also \cite[Corollary 1]{Blasi79}, where De Blasi and Myjak showed a similar
generic
property for continuous and bounded monotone operators on a bounded set.
\begin{corollary}
In $(\MM,\tro)$ there exists a set $G\subset \MM$ which is a countable intersection of open everywhere
dense set in $\MM$ such that each $A\in G$ has
both $R_{A}$ and $J_{A}$ being super-regular. In particular, $A^{-1}(0)\neq\varnothing$ is a singleton.
\end{corollary}
\begin{proof} Observe that for $A, B\in \MM$,
$$\tro(A,B)=\sum_{n=1}^{\infty}\frac{1}{2^n}\frac{\|R_{A}-R_{B}\|_{n}}{1+\|R_{A}-R_{B}\|_{n}}
=\sum_{n=1}^{\infty}\frac{1}{2^n}\frac{2\|J_{A}-J_{B}\|_{n}}{1+2\|J_{A}-J_{B}\|_{n}}=\hro(J_{A},J_{B})$$
by \eqref{e:metricmono} and \eqref{e:metricresol}.
Thus, $(\MM,\tro)$ and $(\JJ,\hro)$ are isometric under the mapping $A\mapsto J_{A}$.
Apply Theorems~\ref{t:resol} to $(\JJ,\hro)$ to obtain $\tilde{G_{1}}\subset\JJ$ such that each $T\in
\tilde{G_{1}}$ is super-regular. This $\tilde{G_{1}}$ corresponds to $G_{1}\subset\MM$ such that
each $A\in G_{1}$ has $J_{A}$ being super-regular and $G_{1}$ is a countable intersection
of open everywhere dense set in $\MM$.
Apply Theorem~\ref{t:mono} to obtain $G_{2}\subset \MM$ such that each $A\in G_{2}$ has $R_{A}$ being
super-regular and $G_{2}$ is a countable intersection
of open everywhere dense set in $\MM$. It suffices to let $G=G_{1}\cap G_{2}$.
\end{proof}

We finish this section with two examples.
\begin{example} \emph{For a maximal monotone operator $A\in\MM$, with regard to super-regularity a variety situations can happen to
$R_{A}$ and $J_{A}$.}

(1) Let $A:X\To X$ be given by $A=N_{\{0\}}$ the normal cone operator.
Then $J_{A}=0$ is super-regular, but $R_{A}=-\Id$ is not super-regular.

(2) Let $A:X\To X$ be given by $A=0$ the zero operator.
Then both $J_{A}=\Id$ and $R_{A}=\Id$ are not super-regular.

(3) Let $A:X\to X$ be given by $A=\Id$. Then $J_{A}=\Id/2$ and $R_{A}=0$ are super-regular.

\end{example}

\begin{example}\label{e:superweak} \emph{A super-regular mapping needs not be contractive.}

Define $T:\RR\rightarrow\RR$ by $T(x)=|\sin x|$ for every $x\in \RR$. Then $T$ is not contractive but super-regular.
$T$ is not contractive because $\sup_{x\in R}|T'(x)|=1.$ To see that $T$ is super-regular, we note that
$0\leq Tx\leq 1$
and
for $n\geq 2$ the ``iterative sequence" $(T^{n})_{n=2}^{\infty}$ satisfies
$$0\leq T^n(x)=\sin (T^{n-1}(x))\leq T^{n-1}(x)\qquad \text{ for every $x\in\RR$}.$$
Being a decreasing monotone sequence bounded below, $(T^{n}(x))_{n=2}^{\infty}$ converges to 0, the unique fixed point
of $T$. Since that the the decreasing function sequence $(T^{n}(x))_{n=1}^{\infty}$ converges to $0$ and that
each $T^{n}$ is continuous, $T^{n}$ converges uniformly to $0$ on every compact subset of $R$ by
the Dini's Theorem \cite{royden}.
\end{example}

%We conclude this section by two remarks.
%\begin{remark}
%Note that the results in \cite{Blasi,zaslavski00,zaslavski01} do not directly apply since there the authors
%consider nonexpansive self mappings on a bounded set.
%\end{remark}
%\begin{remark}
%We do not know whether Theorems~\ref{t:nonexpansive} and
%\ref{t:monotone} still hold when the closures in \eqref{e:almostfixpoint} and \eqref{e:crange} are removed.
The results in the next section indicate that the set of contractive mappings and the set of strongly maximal
monotone operators are too small.
%\end{remark}

\section{Weakly contractive mapping, strong monotonicity and strong firmness}\label{s:weak}
In this section we show that the set of contraction mappings in $(\NON,\rho)$ (a subset
of dense $G_{\delta}$ set in Theorem~\ref{t:reichz}),
the set of strongly monotone mappings in $(\MM,\tro)$
%(a subset of dense $G_{\delta}$ set in Theorem~\ref{t:mono})
 and the set of strongly firm
nonexpansive mappings (a subset of dense $G_{\delta}$ set in Theorem~\ref{t:resol}) are first category,
even they are dense in the corresponding metric spaces.

\begin{definition}
A nonexpansive mapping $T\in\NON$ is weakly contractive if there exists $2>l>0$ such that
\begin{equation}\label{e:weakdef}
\|Tx-Ty\|^2\leq \|x-y\|^2-l(\|x-y\|^2+\scal{x-y}{Tx-Ty})\quad \forall \ x,y\in X.
\end{equation}
\end{definition}
The set of weakly contractive mappings
is strictly larger than the set of contractive mappings since
 $T=-\Id$ is weakly contractive but not contractive.

\begin{definition}
A firmly nonexpansive mapping $T\in\JJ$ is strongly firm nonexpansive if there exists $\varepsilon>0$ such that
$$(1+\varepsilon)T\in \JJ,$$
in particular, $T$ is $1/(1+\varepsilon)$ contractive.
\end{definition}

The following result states the relationship among $R_{A}$ being weakly contractive,
$A$ being strongly monotone  and $J_{A}$ being strongly firmly nonexpansive.
\begin{proposition}\label{p:strongmono}
 Let $A\in\MM$. Then the following are equivalent:
\begin{enumerate}
\item\label{i:weak1} $A$ is strongly monotone for some $\varepsilon>0$;
\item \label{i:weak2} $\varepsilon \Id+(1+\varepsilon) R_{A}$ is nonexpansive;
\item \label{i:weak3} $R_{A}$ is $\frac{2\varepsilon}{1+\varepsilon}$ weakly contractive;
\item \label{i:weak4} $(1+\varepsilon)J_{A}$ is firmly nonexpansive.
\end{enumerate}
\end{proposition}
\begin{proof}
\ref{i:weak1}$\Leftrightarrow$\ref{i:weak2}: \cite[Theorem 4.3]{bmw12}.
\ref{i:weak1}$\Leftrightarrow$\ref{i:weak4}: \cite[Theorem 2.1(xi)]{bmw12}.
\ref{i:weak2}$\Leftrightarrow$\ref{i:weak3}:
\ref{i:weak2} means for $x,y\in X$,
$$\|\varepsilon x+(1+\varepsilon)R_{A}x-(\varepsilon y+(1+\varepsilon)R_{A}y)\|^2\leq \|x-y\|^2,$$
that is,
$$\varepsilon^2\|x-y\|^2+(1+\varepsilon)^2\|R_{A}x-R_{A}y\|^2+2\varepsilon (1+\varepsilon)\scal{x-y}{R_{A}x-R_{A}y}\leq
\|x-y\|^2.$$
Simple algebraic manipulation shows that this is equivalent to
$$\|R_{A}x-R_{A}y\|^2\leq \|x-y\|^2-\frac{2\varepsilon}{1+\varepsilon}(\|x-y\|^2+\scal{x-y}{R_{A}x-R_{A}y}).$$
\end{proof}

\begin{corollary} Assume that $T\in\NON$ is a weakly contraction mapping for some
$0<l<2$. Then $\Fix(T)\neq\varnothing$ and
is a singleton.
\end{corollary}
\begin{proof} By Proposition~\ref{p:strongmono}, $T=R_{A}$ for a maximally monotone mapping and $A$ is strongly maximal monotone. Since $A$ is strongly monotone, we have $\ran A=X$ by Brezis-Haraux's range theorem \cite[Corollary 31.6]{Simons2} so that $A^{-1}(0)\neq\varnothing$ and
$A^{-1}(0)$ is a singleton. The proof is complete by using $\Fix(T)=A^{-1}(0)$.
\end{proof}

\begin{example}
(1) \emph{A weakly contractive mapping needs not be super-regular. }
Let $A:X\To X$ be given by $A=N_{\{0\}}$ the normal cone operator.
Then $R_{A}=-\Id$ weakly contractive but not super-regular.

(2) \emph{A super-regular mapping needs not be weakly contractive.}
From Example~\ref{e:superweak}, the mapping
$$f:\RR\rightarrow\RR: x\mapsto |\sin x|$$ is super-regular.
Since $\frac{f+\Id}{2}$ is not contractive, $f$ is not weakly contractive
by Proposition~\ref{p:strongmono}.
\end{example}

\begin{example} \emph{A nonexpansive mapping can be neither weakly contractive nor super-regular.}
On the Euclidean space $X=\RR^2$, the $\pi/2$-degree rotator $T:X\rightarrow X$ given
$$T=\begin{pmatrix}
0 & -1\\
1 & 0
\end{pmatrix}$$
is neither weakly contractive nor super-regular. Indeed,
$T$ is nonexpansive since $\|Tx\|=\|x\|$ with $x\in X$; $T$ is not weakly nonexpansive because
\eqref{e:weakdef} fails, as
$\langle x-y, Tx-Ty\rangle=0$ for $x,y\in X$; $T$ is not  super-regular because
$||T^{n}x\|=\|x\|$ for every $x\in X, n\in\NN$, and $T^{n}x\not\rightarrow 0$ unless
$x=0$.
\end{example}

The connection between weakly contractive mappings and contractive mappings comes next.
\begin{proposition}\label{p:contraction}
% The following are equivalent:
Let $T\in\NON$.
\begin{enumerate}
\item \label{i:contr1} If
 $T$ is a contraction with modulus $0\leq \beta<1$, i.e.,
$\|Tx-Ty\|\leq \beta \|x-y\|$ for all $x,y\in X,$
then
both $T$ and $-T$ are $(1-\beta)$ weakly contractive.
\item\label{i:contr2} If both $T$ and $-T$ are $(1-\beta)$ weakly contractive,
then $T$ is a contraction with modulus $\sqrt{\beta}$.
\end{enumerate}
\end{proposition}
\begin{proof}
\ref{i:contr1}:
Assume that $T$ is $\beta$ contractive.
For $x,y\in X$, by the Cauchy-Schwartz inequality and $T$ being $\beta$ contractive, we have
$$\scal{x-y}{Tx-Ty}\leq\|x-y\|\|Tx-Ty\|\leq\beta\|x-y\|^2.$$
It follows that
\begin{align}
%&\|x-y\|^2-(1-\sqrt{\beta})(\|x-y\|^2+\scal{x-y}{Tx-Ty})\\
& \|x-y\|^2-(1-\beta)(\|x-y\|^2+\scal{x-y}{Tx-Ty})\\
& \geq \|x-y\|^2-(1-\beta)(\|x-y\|^2+\beta\|x-y\|^2)\\
&=\|x-y\|^2-(1-\beta^2)\|x-y\|^2=\beta^2\|x-y\|^2\\
&\geq \|Tx-Ty\|^2.
\end{align}
Hence $T$ is $(1-\beta)$ weakly contractive. Applying to $-T$, we obtain that $-T$
is $(1-\beta)$ weakly contractive.

\ref{i:contr2}: Assume that both $T$ and $-T$ are $(1-\beta)$ weakly contractive.
Then
$$\|Tx-Ty\|^2\leq \|x-y\|^2-(1-\beta)(\|x-y\|^2+\scal{x-y}{Tx-Ty})\quad \forall \ x,y\in X.$$
$$\|Tx-Ty\|^2\leq \|x-y\|^2-(1-\beta)(\|x-y\|^2-\scal{x-y}{Tx-Ty})\quad \forall \ x,y\in X.$$
Adding these inequality gives
$$2\|Tx-Ty\|^2\leq 2\|x-y\|^2-2(1-\beta)\|x-y\|^2=2\beta\|x-y\|^2$$
which gives $\|Tx-Ty\|\leq \sqrt{\beta}\|x-y\|$ for $x,y\in X$. Hence
$T$ is a $\sqrt{\beta}$ contraction.
\end{proof}

It is very interesting to compare Proposition~\ref{p:strongmono} to Fact~\ref{f:monocontr}.

Our main result in this section is
\begin{theorem}[first category of weakly contraction mappings]\label{t:wcontraction}
 In $(\NON,\rho)$, the set of weak contraction mappings, i.e., $\kK:=$
$$\left\{
\begin{aligned} T\in\NON: &\ \exists\ l>0 \text{ such that }
\|Tx-Ty\|^2\leq \|x-y\|^2-l(\|x-y\|^2+\scal{x-y}{Tx-Ty})\\
& \quad \forall \ x,y\in X.
%\|Tx-Ty\|\leq l \|x-y\|\ \forall \ x,y\in X\right}
\end{aligned}
\right\}
$$
is of first category.
\end{theorem}
\begin{proof}
Let $(l_{n})_{n=1}^{\infty}$ be a positive strictly decreasing sequence in $(0,2)$ with
$\lim_{n\rightarrow\infty}l_{n}=0$. Define
%$$\kK_{n}=\{T\in\NON:\ \exists\ l>0 \text{ such that } \|Tx-Ty\|\leq l_{n} \|x-y\|\ \forall \ x,y\in X\}.$$
\begin{equation}\label{e:nthset}
\kK_{n}:=\left\{
\begin{aligned} T\in\NON: &\
\|Tx-Ty\|^2\leq \|x-y\|^2-l_{n}(\|x-y\|^2+\scal{x-y}{Tx-Ty})\\
& \quad \forall \ x,y\in X.
%\|Tx-Ty\|\leq l \|x-y\|\ \forall \ x,y\in X\right}
\end{aligned}
\right\}.
\end{equation}
Then $\kK_{n+1}\supset \kK_{n}$ for $n\in\NN$ and
$\kK=\bigcup_{n=1}^{\infty}\kK_{n}$. Clearly $\kK_{n}$ is closed in $\NON$.
We show that $\inte\kK_{n}=\varnothing$, where $\inte\kK_{n}$ stands for the interior of $\kK_{n}$.
Let $T\in\kK_{n}$ and $\varepsilon>0$. We will construct $T_{2}\in\NON$ such that
$\rho(T,T_{2})\leq 2\varepsilon$ and $T_{2}\not\in \kK_{n}$. To this end, first apply Lemma~\ref{l:cdense} to find
a contraction map $T_{1}$ with
modulus $0<L<1$ such that
$\rho(T_{1},T)\leq \varepsilon$.  As $T_{1}:X\to X$ is a contraction, it has a fixed point
$x_{0}\in X$. Next we follow the idea from De Blasi and Myjak \cite{Blasi76}.
Put
$$0<\delta:=\frac{(1-L)\varepsilon}{4}<\frac{\varepsilon}{4}.$$
Define
$$\tilde{T_{1}}(x):=\begin{cases}
x & \text{ if $x\in \bB_{\delta}(x_{0})$}\\
T_{1}(x) & \text{ if $x\not\in \bB_{\varepsilon/2}(x_{0})$}.
\end{cases}
$$
Then $\tilde{T_{1}}$ is nonexpansive on $\bB_{\delta}(x_{0})\bigcup (X\setminus \bB_{\varepsilon/2}(x_{0}))$.
To see this, consider three cases: (i) If $x,y\in \bB_{\delta}(x_{0})$, then
$$\|\tilde{T_{1}}(x)-\tilde{T_{1}}(y)\|=\|x-y\|;$$
(ii) If $x,y\not\in \bB_{\varepsilon/2}(x_{0})$, then
$$\|\tilde{T_{1}}(x)-\tilde{T_{1}}(y)\|=\|T_{1}x-T_{1}y\|\leq L\|x-y\|;$$
(iii) $x\in \bB_{\delta}(x_{0}), y\not\in \bB_{\varepsilon/2}(x_{0})$. Note that $T_{1}x_{0}=x_{0}$,
$T_{1}$ being contractive with modulus $L$, and
\begin{align}
\|x-y\|&=\|x-x_{0}+x_{0}-y\|\geq \|x_0-y\|-\|x-x_{0}\|\label{e:distance0}\\
&\geq \frac{\varepsilon}{2}-\|x-x_{0}\|\geq \frac{\varepsilon}{2}-\delta\\
&=\frac{\varepsilon}{2}-\frac{(1-L)\varepsilon}{4}=\frac{(1+L)\varepsilon}{4}.\label{e:distance}
\end{align}
It follows that
\begin{align}
\|\tilde{T_{1}}(x)-\tilde{T_{1}}(y)\| &= \|x-T_{1}y\|\\
&=\|x-x_{0}+T_{1}x_{0}-T_{1}x+T_{1}x-T_{1}y\|\\
& \leq \|x-x_{0}\|+\|T_{1}x_{0}-T_{1}x\|+\|T_{1}x-T_{1}y\|\\
& \leq \|x-x_{0}\|+L\|x-x_{0}\|+L\|x-y\|\\
& =(1+L)\|x-x_{0}\|+L\|x-y\|\\
&\leq (1+L)\delta+L\|x-y\|\\
&=(1+L)\frac{(1-L)\varepsilon}{4}+L\|x-y\| \\
&=(1-L)\frac{(1+L)\varepsilon}{4}+L\|x-y\| \quad (\text{using }\eqref{e:distance0}-\eqref{e:distance})\\
& \leq (1-L)\|x-y\|+L\|x-y\|=\|x-y\|.
\end{align}
According to the Kirszbraun-Valentine extension theorem, see, e.g., \cite{reich05},
there exists a nonexpansive mapping $T_2:X\to X$ extending $\tilde{T_{1}}$ from $\dom \tilde{T_{1}}$ to
$X$.

{\sl Claim 1:} $\rho(T_{2},T_{1})\leq \varepsilon$.

To see this, observe that $T_{2}x=\tilde{T_{1}}(x)=T_{1}(x)$ if $x\in X\setminus \bB_{\varepsilon/2}(x_{0})$.
When $x\in \bB_{\delta}(x_{0})$ we have
\begin{align}
\|T_{2}x-T_{1}(x)\| & =\|x-T_{1}x\|=\|x-x_{0}+T_{1}x_{0}-T_{1}x\|\\
& \leq \|x-x_{0}\|+\|T_{1}x_{0}-T_{1}x\|
\leq \|x-x_{0}\|+L\|x-x_{0}\|\\
&=(1+L)\|x-x_{0}\|\leq (1+L)\delta=(1+L) \frac{(1-L)\varepsilon}{4}\leq \varepsilon;
\end{align}
When $x\in \bB_{\varepsilon/2}(x_{0})\setminus \bB_{\delta}(x_{0})$, pick
$$y:=x_{0}+\frac{\varepsilon}{2}\frac{y-x_{0}}{\|y-x_{0}\|}$$
so that $y\in \bB_{\varepsilon/2}(x_{0})$ and $T_{2}y=T_{1}y$. We have
\begin{align}
\|T_{2}x-T_{1}x\| &=\|T_{2}x-T_{1}x-(T_{2}y-T_{1}y)\|=\|(T_{2}x-T_{2}y)-(T_{1}x-T_{1}y)\|\\
&\leq \|T_{2}x-T_{2}y\|+\|T_{1}x-T_{1}y\|\\
&\leq \|x-y\|+L\|x-y\|=(1+L)\|x-y\|\leq (1+L)(\varepsilon/2-\delta)\leq\varepsilon.
\end{align}
Then
$$\rho(T,T_{2})\leq \rho(T,T_{1})+\rho(T_{1},T_{2})\leq 2\varepsilon.$$

{\sl Claim 2:} $T_{2}\not\in \kK_{n}$. This is because $T_{2}x=x$ for $x\in \bB_{\delta}(x_{0})$.

Since $\varepsilon$ was arbitrary, $\inte \kK_{n}=\varnothing$.
This completes the proof.
\end{proof}

Combing Theorem~\ref{t:wcontraction} and Proposition~\ref{p:contraction}(i) immediately yields
\begin{corollary}[first category of contraction mappings]\label{c:unbound} In $(\NON,\rho)$, the set of contractive mappings, i.e., $\kK=$
$$\left\{
%\begin{aligned}
T\in\NON: \ \exists\ 1>l\geq 0 \text{ such that }
\|Tx-Ty\|\leq l\|x-y\|\  \forall \ x,y\in X
%\|Tx-Ty\|\leq l \|x-y\|\ \forall \ x,y\in X\right}
%\end{aligned}
\right\}
$$
is of first category.
\end{corollary}
While a similar result for nonexpansive mappings
defined on a closed bounded convex set $C\subset X$ was obtained by
De Blasi and Myjak in \cite{Blasi76} and Reich \cite{reichcom}, Corollary~\ref{c:unbound}
concerns nonexpansive mappings on
a unbounded set $X$.

There are many ways to generate strongly monotone mappings: $A+\varepsilon\Id$ (Tychonov regularization),
$S_{A}^{\varepsilon}=\cR(A,\Id, 1-\varepsilon,\varepsilon)$ (self-dual regularization), see, e.g., \cite{wang11}.
Corresponding results for maximal monotone operators and firmly nonexpansive mappings follow at once by
combing Proposition~\ref{p:strongmono} and Theorem~\ref{t:wcontraction}.

\begin{corollary}[first category of strongly monotone mappings] In $(\MM,\tro)$, the set
$$\{A\in\MM:\ \exists \ \varepsilon> 0 \text{ such that $A$ is $\varepsilon$ strongly monotone}\}$$
is of first category.
\end{corollary}

\begin{corollary}[first category of strongly firm nonexpansive mappings] In $(\JJ,\hro)$, the set
$$\{T\in\JJ:\ \exists\ \varepsilon >0  \text{ such that $(1+\varepsilon)T$ is firmly nonexpansive}\}$$
is of first category.
\end{corollary}

\section*{Appendix}
For $C\subset X$, $\overline{C}$ denotes its norm closure. The proofs to Theorems~\ref{t:unified}, \ref{t:resol} and
\ref{t:mono} are harder, and rely on Reich and Zaslavaski's super-regularity mappings. If one only
wants $0\in \overline{\ran(\Id-T)}$, $0\in\overline{\ran A}$ and asymptotic regularity of $J_{A}$ (much weaker results),
a much
simpler argument works. This is the purpose of this appendix.
\begin{theorem}[almost fixed point of nonexpansive mapping]\label{t:nonexpansive} The set
\begin{equation}\label{e:almostfixpoint}
G:=\{T:\ 0\in\overline{\ran (\Id- T)}\}
\end{equation}
is dense $G_{\delta}$ in $(\NON,\rho)$. Thus, generically nonexpansive mappings almost have fixed points.
\end{theorem}
\begin{proof} For every $n\in\NN$ define
$$O_{n}:=\left\{T\in\NON:\ \text{ there exists $x\in X$ such that } \|x-Tx\|<\frac{1}{n}\right\}.$$
{\sl Claim 1.} $O_{n}$ is dense. Let $T\in \NON$ and $\varepsilon>0$. Apply Lemma~\ref{l:cdense}
to find a contraction $T_{2}$ such that $\rho(T,T_{2})<\varepsilon$.
Since $T_{2}$ is a contraction, it has a fixed point
by the Banach Contraction Principle \cite[Theorem 5.1.2]{kreyzig}, thus
$T_{2}\in O_{n}$.
%By the definition of $\rho$,
Therefore, $O_{n}$ is dense in $\NON$.

\noindent{\sl Claim 2.} $O_{n}$ is open. Let $T\in O_{n}$. Then there exists $x\in X$ such that
\begin{equation}\label{e:x}
\|x-Tx\|<\frac{1}{n}.
\end{equation}
Assume that $K\in \NN$ and $\|x\|<K$. Put
$$r=\frac{1}{2^{K}}\frac{1/n-\|x-Tx\|}{1+1/n-\|x-Tx\|}.$$
We show that $\bB_{r}(T):=\{T_{1}\in \NON: \rho(T_{1},T)<r\}\subset O_{n}$. Let $T_{1}\in \bB_{r}(T)$.
Since $\rho(T_{1},T)<r$, we have
$$\frac{1}{2^K}\frac{\|T_{1}-T\|_{K}}{1+\|T_{1}-T\|_{K}}\leq \rho(T_{1},T)<\frac{1}{2^{K}}\frac{1/n-\|x-Tx\|}{1+1/n-\|x-Tx\|}$$
so that
$$\frac{\|T_{1}-T\|_{K}}{1+\|T_{1}-T\|_{K}}<\frac{1/n-\|x-Tx\|}{1+1/n-\|x-Tx\|}.$$
It follows that
$$\|T_{1}-T\|_{K}<\frac{1}{n}-\|x-Tx\|.$$
Then using $\|x\|\leq K$,
\begin{align}
\|x-T_{1}x\|& =\|x-Tx+Tx-T_{1}x\|\leq \|x-Tx\|+\|Tx-T_{1}x\|\\
& \leq \|x-Tx\|+\|T-T_{1}\|_{K}<\|x-Tx\|+\frac{1}{n}-\|x-Tx\|=\frac{1}{n}.
\end{align}
Therefore $T_{1}\in O_{n}$. Since $T_{1}\in \bB_{r}(T)$ was arbitrary, $\bB_{r}(T)\subset O_{n}$.

As $(\NON,\rho)$ is a complete metric space, $\bigcap_{n=1}^{\infty}O_{n}$ is a dense $G_{\delta}$ set
in $\NON$ by Fact~\ref{f:baire}.
If $T\in \bigcap_{n=1}^{\infty}O_{n}$, then for every $n\in\NN$ there exists $x_{n}\in X$ such that
$$\|x_{n}-Tx_{n}\|<\frac{1}{n}$$
thus $0\in\overline{\ran(\Id-T)}$.
%The proof is complete by noting that
Hence $\bigcap_{n=1}^{\infty}O_{n}\subset G$.
On the other hand, if $T\in G$, then $0\in \overline{\ran (\Id -T)}$. It follows that for every $n$ there exists
$x_{n}\in X$ such that $\|x_{n}-Tx_{n}\|<1/n$ so $T\in O_{n}$. As this holds for every $n$ and $T\in G$, we have
$G\subset \bigcap_{n=1}^{\infty}O_{n}$. Altogether, $G=\bigcap_{n=1}^{\infty} O_{n}$ which is a dense
$G_{\delta}$ set in $\NON$. This completes the
proof.
%, therefore $G$ is residual.
\end{proof}

\begin{theorem}[almost zeros of maximal monotone operator]\label{t:monotone}
 In $(\MM,\tro)$, the set
\begin{equation}\label{e:crange}
\{A\in\MM:\ 0\in\overline{\ran A}\}
\end{equation}
is a dense $G_{\delta}$ set. Hence, generically  maximally monotone operators almost have zeros.
\end{theorem}
\begin{proof}
By Theorem~\ref{t:nonexpansive} and Proposition~\ref{p:relation}(i), the set
$$\{A\in \MM: 0\in\overline{\ran(\Id-R_{A})}\}$$
is dense $G_{\delta}$ in $(\MM,\tro)$. Observe that
$$\ran (\Id-R_{A})=\ran (2\Id -2J_{A})=2\ran(\Id-J_{A})=2\ran J_{A^{-1}}=2\dom A^{-1}=2\ran A.$$
Hence \eqref{e:crange} holds.
\end{proof}

%On $\JJ$ let us define a metric
%$$\hro(T_{1},T_{2}):=\rho(2T_{1}-\Id, 2T_{2}-\Id)=\sum_{n=1}^{\infty}\frac{1}{2^n}\frac{\|2T_{1}-2T_{2}\|_{n}}{1+\|2T_{1}-2T_{2}\|_{n}}$$
%for $T_{1},T_{2}\in \JJ$.
%Then $(\JJ,\hro)$ is a complete metric space, since by Fact~\ref{f:firm} under the mapping $T\mapsto 2T-\Id$
%\begin{equation}\label{e:isometry}
%(\JJ,\hro)\text{ and } (\NON,\rho) \text{ are isometric}
%\end{equation}
%and $(\NON,\rho)$ is complete by Proposition~\ref{p:setup}.

Recall that $T:X\to X$ is \emph{asymptotically regular} at $x$ if $\lim_{n\rightarrow\infty}
(T^{n+1}x-T^{n}x)=0$,
cf. \cite{bruck,BC2011}. Asymptotic regularity is one of critical properties in many iterative algorithms,
\cite[page 79]{BC2011}, \cite{baillon},

\begin{theorem}[asymptotic regularity of resolvent]\label{t:resolvent}
 In $(\JJ,\hro)$, the set
\begin{equation}\label{e:difference}
\{T\in \JJ:\ \|T^{n+1}x-T^{n}x\|\rightarrow 0 \ \forall x\in X\}
\end{equation}
is a dense $G_{\delta}$ set. Consequently, generically resolvents are asymptotically regular.
\end{theorem}
\begin{proof}
Each $T\in \JJ$ is firmly nonexpansive, so strongly nonexpansive.
By \cite[Corollary 1.5]{bruck} Bruck and Reich,
\begin{equation}\label{e:br}
\lim_{n\rightarrow \infty} (T^{n}x-T^{n+1}x)=v
\end{equation}
where $v$ is the smallest norm element of $\overline{\ran(\Id-T)}$. It follows for Theorem~\ref{t:nonexpansive}
and \eqref{e:isometry}
that the set
$$\{T\in\JJ:\ 0\in \overline{\ran(\Id-(2T-\Id))}\}$$
i.e., $$\{T\in \JJ:\ 0\in \overline{\ran(\Id-T)}\}$$ is a dense $G_{\delta}$ set in
$\JJ$. It suffices to apply \eqref{e:br}.
\end{proof}

\section*{Acknowledgments}
I would like to thank Dr. Heinz Bauschke for his constructive suggestions and comments on the paper.
%was partially supported by the Natural Sciences and
%Engineering Research Council of Canada and by the Canada Research Chair
%Program.
Xianfu Wang was partially
supported by the Natural Sciences and Engineering Research Council
of Canada.

%\newpage


\small

\begin{thebibliography}{999}
\bibitem{attouch} H. Attouch, \emph{Variational Convergence for Functions and Operators},
Applicable Mathematics Series, Pitman Advanced Publishing Program, Boston, MA, 1984.

\bibitem{baillon} J.B. Baillon, R.E. Bruck and S. Reich, On the asymptotic behavior of nonexpansive mappings and semigroups in Banach spaces, \emph{Houston J. Math.} 4 (1978), 1--9.

%\bibitem{ABC}
%H.\ Attouch, L.M.\ Brice\~{n}o-Arias, and P.L.\ Combettes,
%A parallel splitting method for coupled monotone inclusions,
%\emph{SIAM Journal on Control and Optimization}~48
%(2010), 3246--3270.

%\bibitem{BBR}
%J.B.\ Baillon, R.E.\ Bruck, and S.\ Reich,
%On the asymptotic behavior of nonexpansive mappings
%and semigroups in Banach spaces,
%\emph{Houston Journal of Mathematics} 4 (1978), 1--9.

%\bibitem{BH}
%J.-B.\ Baillon and G.\ Haddad,
%Quelques propri\'et\'es des op\'erateurs angle-born\'es et
%$n$-cycliquement monotones,
%\emph{Israel Journal of Mathematics} 26 (1977), 137--150.

\bibitem{BBBRW}
S.\ Bartz, H.H.\ Bauschke, J.M.\ Borwein,
S.\ Reich, and X.\ Wang,
Fitzpatrick functions, cyclic monotonicity and Rockafellar's
antiderivative,
\emph{Nonlinear Anal.}~66 (2007), 1198--1223.

\bibitem{moursi12} H.H. Bauschke, R.I. Bot, W.L. Hare, W.M. Moursi,
Attouch-Th\'era duality revisited: Paramonotonicity and operator splitting,
\emph{J. Approx. Theory}~164 (2012), 1065--1084.

%\bibitem{BBC}
%H.H.\ Bauschke, J.M.\ Borwein, and P.L.\ Combettes,
%Essential smoothness, essential strict convexity,
%and Legendre functions in Banach spaces,
%\emph{Communications in Contemporary Mathematics} 3 (2001), 615--647.

%\bibitem{BBL}
%H.H.\ Bauschke, J.M.\ Borwein, and A.S.\ Lewis,
%The method of cyclic projections for closed convex sets in Hilbert
%space, in
%\emph{Recent Developments in Optimization Theory and Nonlinear
%Analysis (Jerusalem 1995)},
%Y.\ Censor and S.\ Reich (editors),
%Contemporary Mathematics~vol.~204,
%American Mathematical Society, % Providence, R.I.,
%pp.~1--38, 1997.

%\bibitem{BC2010}
%H.H.\ Bauschke and P.L.\ Combettes,
%The Baillon-Haddad theorem revisited,
%\emph{Journal of Convex Analysis} 17 (2010), 781--787.
%
\bibitem{BC2011}
H.H.\ Bauschke and P.L.\ Combettes,
\emph{Convex Analysis and Monotone Operator Theory in Hilbert Spaces},
Springer, 2011.

\bibitem{bmw12}
H.H. Bauschke, S.M. Moffat and X. Wang, Firmly nonexpansive mappings and maximally monotone operators: correspondence and duality, \emph{Set-Valued Var. Anal.}~20 (2012), 131--153.



\bibitem{Blasi} F.S. De Blasi and J. Myjak, Sur la porosit\'e de l'ensemble des contractions sans point fixe,
(French) [On the porosity of the set of contractions without fixed points], \emph{ C. R. Acad. Sci. Paris S\'er. I Math.}~308 (1989), 51--54.

\bibitem{Blasi79} F.S. De Blasi and J. Myjak, Generic properties of contraction semigroups
and fixed points of nonexpansive operators, \emph{Proc. Amer. Math. Soc.}~77 (1979), 341--347.

\bibitem{Blasi76} F.S. De Blasi and J. Myjak, Sur la convergence des approximations successives pour
les contractions non lin\'eaires dans un espace de Banach, \emph{C. R. Acad. Sci. Paris}~283 (1976), 185--187.

\bibitem{bruck} R.E. Bruck and S. Reich,  Nonexpansive projections and resolvents of accretive operators in Banach spaces,
\emph{Houston J. Math.}~3 (1977), 459--470.

%%\bibitem{BauEd}
%%H.H.\ Bauschke and M.R.\ Edwards,
%%A conjecture by De Pierro is true for translates of regular
%%subspaces,
%%\emph{Journal of Nonlinear and Convex Analysis}~6 (2005), 93--116.
%
%\bibitem{BWW}
%H.H.\ Bauschke, X.\ Wang, and C.J.S.\ Wylie,
%Fixed points of averages of resolvents: geometry and algorithms,
%\texttt{http://arxiv.org/pdf/1102.1478v1}, February 2011.
%
%\bibitem{BWWSIOPT}
%H.H.\ Bauschke, X.\ Wang, and C.J.S.\ Wylie,
%Fixed points of averages of resolvents: geometry and algorithms,
%\emph{SIAM Journal on Optimization} 22 (2012), 24--40.
%
%\bibitem{bglw08}
%H.H. Bauschke, R. Goebel, Y. Lucet, and X. Wang, The proximal average: basic theory,
%\emph{SIAM Journal on Optimization} 19 (2008), no. 2, 766--785.

%\bibitem{Bjorck}
%{\AA}.\ Bj{\"o}rck,
%\emph{Numerical Methods for Least Squares Problems},
%SIAM, 1996.

%\bibitem{BorBor}
%J.M.\ Borwein and P.B.\ Borwein,
%\emph{Pi and the AGM},
%Wiley, New York, 1987.
%
%\bibitem{berman}
%A. Berman and R.J. Plemmons,
%\emph{Nonnegative Matrices in the Mathematical Sciences},
%SIAM, 1994.
%
%%\bibitem{BorVanBook}
%%J.M.\ Borwein and J.D.\ Vanderwerff,
%%\emph{Convex Functions},
%%Cambridge University Press, 2010.


%\bibitem{Brezis}
%H. Br\'ezis,
%\emph{Operateurs Maximaux Monotones et
%Semi-Groupes de Contractions dans les Espaces de Hilbert},
%North-Holland/Elsevier, 1973. % New York

%\bibitem{B-AC}
%L.M.\ Brice\~{n}o-Arias and P.L.\ Combettes,
%Convex variational formulation with smooth coupling
%for multicomponent signal decomposition and recovery,
%\emph{Numerical Mathematics: Theory, Methods, and Applications}~2
%(2009), 485--508.
\bibitem{combettes02}
P.L. Combettes and T.  Pennanen,  Generalized Mann iterates for constructing fixed points in Hilbert spaces,
\emph{J. Math. Anal. Appl.} 275 (2002), no. 2, 521--536.

\bibitem{Comb04}
P.L.\ Combettes,
Solving monotone inclusions via compositions of nonexpansive averaged
operators,
\emph{Optimization}~53 (2004), 475--504.

\bibitem{eckstein92} J. Eckstein and D.P. Bertsekas, On the Douglas-Rachford splitting method and
the proximal point algorithm for maximal monotone operators, \emph{Math. Programming} 55 (1992), 293--318.

\bibitem{GR}
K.\ Goebel and S.\ Reich,
\emph{Uniform Convexity, Hyperbolic Geometry, and Nonexpansive Mappings},
Marcel Dekker, 1984.

\bibitem{GK}
K.\ Goebel and W.A.\ Kirk,
\emph{Topics in Metric Fixed Point Theory},
Cambridge University Press, 1990.

\bibitem{kreyzig} E. Kreyszig, \emph{Introductory Functional Analysis with Applications},
John Wiley \& Sons, New-York, 1978.

\bibitem{lions}
P.L. Lions and B. Mercier, Splitting algorithms for the sum of two nonlinear operators,
\emph{SIAM J. Numer. Anal.}~16  (1979), no. 6, 964--979.

\bibitem{Minty} G.J.\ Minty, Monotone (nonlinear) operators in Hilbert space,
\emph{Duke Math. J.}, 29 (1962), 341--346.

\bibitem{reichcom}
S. Reich, Genericity and porosity in nonlinear analysis and
optimization, \emph{ESI Preprint 1756, 2005}, Proceedings of CMS'05 (Computer
Methods and Systems), Krakow, 2005, 9--15.

\bibitem{reichza01} S. Reich and A. J. Zaslavski,
Generic aspects of metric fixed point
theory,
\emph{Handbook of Metric Fixed Point Theory}, Kluwer, Dordrecht, 2001,
557--575.

\bibitem{reichmath00} S. Reich and A. J. Zaslavski,
Convergence of Krasnoselskii-Mann iterations of nonexpansive operators,
\emph{Math. Comput. Modelling}~32 (2000), 1423--1431.

\bibitem{reich05} S. Reich and S. Simons,
Fenchel duality, Fitzpatrick functions and the Kirszbraun-Valentine extension theorem,
\emph{Proc. Amer. Math. Soc.}~133 (2005), 2657--2660.

\bibitem{zaslavski00} S. Reich and A.J. Zaslavski,
Almost all nonexpansive mappings are contractive,
\emph{C. R. Math. Acad. Sci. Soc. R. Can.}~22 (2000), 118--124.

\bibitem{zaslavski01} S. Reich and A.J. Zaslavski, The set of noncontractive mappings is $\sigma$-porous in the space of all nonexpansive mappings, \emph{C. R. Acad. Sci. Paris Sér. I Math.}~333 (2001), 539--544.

\bibitem{rockprox}R.T.\ Rockafellar, Monotone operators and the
proximal point algorithm, \emph{SIAM J. Control Optim.}, 14 (1976), 877--898.

\bibitem{royden} H.L. Royden, \emph{Real Analysis}, Prentice Hall, 3rd edition,  1988.

%\bibitem{davidson} K.R. Davidson and A.P. Donsig, \emph{Real Analysis and Applications:
%Theory in Practice}, Springer, New York, 2010.

%\bibitem{Browder67}
%F.E.\ Browder,
%Convergence theorems for sequences of nonlinear operators
%in Banach spaces,
%\emph{Mathematische Zeitschrift}~100 (1967), 201--225.

%\bibitem{BR}
%R.E.\ Bruck and S.\ Reich,
%Nonexpansive projections and resolvents of accretive operators
%in Banach spaces,
%\emph{Houston Journal of Mathematics}~3 (1977), 459--470.
%

%\bibitem{BurIus}
%R.S.\ Burachik and A.N.\ Iusem,
%\emph{Set-Valued Mappings and Enlargements
%of Monotone Operators},
%Springer, 2008.

%\bibitem{Byrne05}
%C.L.\ Byrne,
%\emph{Signal Processing},
%AK Peters, 2005.

%\bibitem{Byrne08}
%C.L.\ Byrne,
%\emph{Applied Iterative Methods},
%AK Peters, 2008.

%\bibitem{CEG}
%Y.\ Censor, P.P.B.\ Eggermont, and D.\ Gordon,
%Strong underrelaxation in Kaczmarz's method for inconsistent
%systems,
%\emph{Numerische Mathematik}~41 (1983), 83--92.

%\bibitem{CIZ}
%Y. Censor, A.N.\ Iusem, and S.A.\ Zenios,
%An interior point method with Bregman functions
%for the variational inequality problem with
%paramonotone operators,
%\emph{Mathematical Programming Series~A}~81 (1998), 373--400.

%\bibitem{Comb94}
%P.L.\ Combettes,
%Inconsistent signal feasibility problems:
%least-squares solutions in a product space,
%\emph{IEEE Transactions on Signal Processing}~42 (1994), 2955--2966.

%\bibitem{Comb04}
%P.L.\ Combettes,
%Solving monotone inclusions via compositions of nonexpansive averaged
%operators,
%\emph{Optimization}~53 (2004), 475--504.

%\bibitem{Cross}
%R.\ Cross,
%\emph{Multivalued Linear Operators},
%Marcel Dekker, 1998.

%\bibitem{DeP}
%A.R.\ De Pierro,
%From parallel to sequential projection methods and vice versa
%in convex feasibility: results and conjectures, in
%\emph{Inherently Parallel Algorithms in Feasibility and
%Optimization and Their Applications (Haifa 2000)},
%D.\ Butnariu, Y.\ Censor, and S.\ Reich (editors),
%Elsevier, %Amsterdam, The Netherlands
%pp.~187--201, 2001.
%
%\bibitem{EckBer}
%J.\ Eckstein and D.P.\ Bertsekas,
%\emph{On the Douglas-Rachford splitting method
%and the proximal point algorithm for maximal monotone
%operators},
%\emph{Mathematical Programming Series A} 55 (1992), 293--318.

%\bibitem{Hackbusch}
%W.\ Hackbusch,
%\emph{Iterative Solution of Large Sparse Systems of Equations},
%Springer-Verlag, 1994.

%\bibitem{Johnson}
%R.A.\ Johnson,
%\emph{Advanced Euclidean Geometry},
%Dover Publications,  % New York
%1960.

%\bibitem{Krey}
%E.\ Kreyszig,
%\emph{Introductory Functional Analysis with Applications},
%Wiley, 1989.

%\bibitem{Horn}
%R.A. Horn and C.R. Johnson, \emph{Matrix Analysis}, Cambridge University Press, Cambridge, 1985.
%
%\bibitem{Krafft}
%O. Krafft and M. Schaefer, Convergence of the powers of a circulant stochastic matrix,
%\emph{Linear Algebra and its Applications}~127 (1990), 59--69.
%
%\bibitem{Megg}
%R.E.\  Megginson,
%\emph{An Introduction to Banach Space Theory},
%Springer-Verlag, 1998.
%
%\bibitem{Meyer}
%C.D.\ Meyer,
%\emph{Matrix Analysis and Applied Linear Algebra},
%SIAM, 2000.

%\bibitem{Minty}
%G.J.\ Minty,
%Monotone (nonlinear) operators in Hilbert spaces,
%\emph{Duke Mathematical Journal} 29 (1962), 341--346.


%\bibitem{Moreau}
%J.-J.\ Moreau,
%Proximit\'e et dualit\'e dans un espace hilbertien,
%\emph{Bulletin de la Soci\'et\'e Math\'ematique de France} 93 (1965),
%273--299.

%\bibitem{Opial67}
%Z.\ Opial,
%Weak convergence of the sequence of successive approximations
%for nonexpansive mappings,
%\emph{Bulletin of the American Mathematical Society}~73 (1967), 591--597.

%\bibitem{Ortega}
%J.M. Ortega,
%\emph{Numerical Analysis:} A Second Course,
%second edition, SIAM, 1990.
%
%\bibitem{Ostrowski}
%A.M.\ Ostrowski,
%\emph{Solution of Equations and Systems of Equations},
%second edition, Academic Press, New York and London, 1966.
%
%
%
%\bibitem{Prasolov}
%V.V.\ Prasolov,
%\emph{Polynomials},
%Springer-Verlag, Berlin, 2004.
%
%
%%\bibitem{Rock70}
%%R.T.\ Rockafellar,
%%\emph{Convex Analysis},
%%Princeton University Press, Princeton, 1970.
%
%%\bibitem{Rock76}
%%R.T.\ Rockafellar,
%%Monotone operators and the proximal point algorithm,
%%\emph{SIAM Journal on Control and Optimization} 14
%%(1976), 877--898.
%
\bibitem{Rock98}
R.T.\ Rockafellar and R. J-B\ Wets,
\emph{Variational Analysis},
Springer, %New York,
corrected 3rd printing, 2009.

%\bibitem{Saad}
%Y.\ Saad,
%\emph{Iterative Methods for Sparse Linear Systems},
%2nd edition, SIAM, 2003.

%\bibitem{Simons1}
%S.\ Simons,
%\emph{Minimax and Monotonicity},
%Springer-Verlag,
%1998.

\bibitem{Sabach11} S. Sabach, Products of finitely many resolvents
of maximal monotone mappings in
reflexive Banach spaces, \emph{SIAM J. Optim.} 21 (2011), 1289--1308.

\bibitem{Simons2}
S.\ Simons,
\emph{From Hahn-Banach to Monotonicity},
Springer,
2008.

%\bibitem{WangBau}
%X.\ Wang and H.H.\ Bauschke,
%Compositions and averages of two resolvents:
%relative geometry of fixed point sets and
%a partial answer to a question by C.~Byrne,
%\emph{Nonlinear Analysis: Theory, Methods \& Applications}~74 (2011),
%4550--4572.

\bibitem{wang11} X. Wang, Self-dual regularization of monotone operators via the resolvent average,
\emph{SIAM J. Optim.}~21 (2011), 438--462.

%\bibitem{Zalinescu}{C.\ Z\u{a}linescu},
%\emph{Convex Analysis in General Vector Spaces},
%World Scientific Publishing, 2002.

%\bibitem{Zara}
%E.H.\ Zarantonello,
%Projections on convex sets in Hilbert space and spectral theory I.
%Projections on convex sets, in
%\emph{Contributions to Nonlinear Functional Analysis},
%E.H.\ Zarantonello (editor), pp.~237--341, Academic Press, 1971.


%\bibitem{Zeidler2a}
%E.\ Zeidler,
%\emph{Nonlinear Functional Analysis and Its Applications II/A:
%Linear Monotone Operators},
%Springer-Verlag, 1990.

\bibitem{Zeidler2b}
E.\ Zeidler,
\emph{Nonlinear Functional Analysis and Its Applications II/B:
Nonlinear Monotone Operators},
Springer-Verlag, 1990.

%\bibitem{Zeidler1}
%E.\ Zeidler,
%\emph{Nonlinear Functional Analysis and Its Applications I:
%Fixed Point Theorems},
%Springer-Verlag, 1993.

\end{thebibliography}
\end{document}